\documentclass[letterpaper, 10 pt, conference]{ieeeconf}  

\IEEEoverridecommandlockouts                              
\usepackage{graphics} 
\usepackage{amsmath} 
\usepackage{amssymb}  
\usepackage{color}
\usepackage{comment}

\usepackage{amsthm}
\usepackage{cite}

\newcommand{\name}[0]{MRPC}

\newtheorem{prop}{Proposition}
\newtheorem{thm}{Theorem}
\newtheorem{lemma}{Lemma}

\title{\LARGE
Thinking Fast and Slow: Optimization Decomposition Across Timescales
}

\author{Gautam Goel \and Niangjun Chen \and Adam Wierman
\thanks{Gautam Goel, Niangjun Chen, and Adam Wierman are with the Department of Computing and Mathematical Sciences, California Institute of Technology. Email:{\tt\small \{ggoel, ncchen, adamw\}@caltech.edu}
This work was supported by the NSF through grants AitF-1637598, CNS-1518941, CPS-154471, CNS-1319820, EPAS-1307794. We also thank Desmond Cai and Nik Matni for helpful discussions.}
}%

\begin{document}

\maketitle
\thispagestyle{empty}
\pagestyle{empty}

\begin{abstract}
Many real-world control systems, such as the smart grid and human sensorimotor control systems, have decentralized components that react quickly using local information and centralized components that react slowly using a more global view. This paper seeks to provide a theoretical framework for how to design controllers that are decomposed across timescales in this way.  The framework is analogous to how the network utility maximization framework uses optimization decomposition to distribute a global control problem across independent controllers, each of which solves a local  problem; except our goal is to decompose a global problem temporally, extracting a timescale separation.  Our results highlight that decomposition of a multi-timescale controller into a fast timescale, reactive controller and a slow timescale, predictive controller can be near-optimal in a strong sense. In particular, we exhibit such a design, named Multi-timescale Reflexive Predictive Control (\name), which maintains a per-timestep cost within a constant factor of the offline optimal in an adversarial setting. 
\end{abstract}

\section{INTRODUCTION}

Modern control systems nearly always operate at multiple timescales.  In the power grid, slow timescale economic dispatch is used to determine which baseload generators will supply power, while fast timescale frequency regulation is used to correct any imbalance between demand and supply that may arise \cite{cai2015distributed}.  In networking, software defined networks use a slow timescale ``control plane'' controller to decide where to send data packets, whereas fast timescale ``data plane'' controllers are responsible for routing the actual data \cite{kreutz2015software}.  Even human sensorimotor control exhibits the same phenomenon, with slow timescale behaviors such as trajectory planning and fast timescale behaviors such as involuntary reflexes \cite{yamashita2008emergence,smith2006interacting,fusi2007neural, milton2011delayed}.  In fact, such timescale separation has consequently been proposed for the control of robotic systems \cite{yamashita2008emergence,espenschied1996biologically}.

Thus, the design and analysis of multi-timescale control systems has received considerable attention.  However, the design of control policies for multi-timescale control systems typically does not address the joint problem of designing control policies across timescales.  Instead, controllers for each timescale are designed independently.  For example, in the power grid, the slow timescale problem of economic dispatch is usually studied separately from the fast timescale problem of frequency regulation.  Similarly, in software defined networking, the design of the control plane and data plane controllers are usually considered separately.

Across these and other applications timescale separation is \textit{assumed} rather than derived, and the resulting subproblems are then studied independently, without guarantees about how they operate jointly. As a result, there are significant inefficiencies that are inherent to the resulting designs, even if each timescale problem is solved optimally.  For example, recent work jointly designing economic dispatch and frequency regulation in the the power grid highlights significant inefficiency in designs that treated the two timescales independently  \cite{cai2015distributed}.  

In this paper, our goal is to develop a framework for \emph{deriving} rather than \textit{assuming} a timescale separation in global optimization problems.  In particular, we adapt the idea of optimization decomposition from the domain of distributed control into the domain of multi-timescale control.  

There is a vast literature on optimization decomposition, in fields as diverse as Internet congestion control \cite{srikant2012mathematics,low2002internet}, smart grid control \cite{deeb1990linear,cai2015distributed}, robotics \cite{brooks1986robust, espenschied1996biologically} and beyond  \cite{chiang2007layering}. The idea of this approach is to decompose a global optimization problem into smaller localized subproblems, each of which is solved by independent controllers. See \cite{chiang2007layering} for a survey.  In a similar way, our goal in this paper is to look for decompositions of a global optimization problem in \emph{time}, as opposed to in \emph{space}.

However, this goal is made challenging by the tight coupling between the timescales due to the underlying dynamics of the system under consideration.  Typically, spatial optimization decomposition is performed for static optimizations, but in multi-timescale control the dynamics of the system cannot be ignored.  Any slow timescale action will impact the future state via the dynamics and hence must be taken into account when designing the fast controller; conversely, any fast timescale action impacts the state seen by a slow controller and thus impacts its design as well. This makes it unclear whether it is possible to achieve a clean separation between controllers at different timescales. 

\subsection{Contributions of this paper}  We make three main contributions in this paper.  

Firstly, we introduce a simple but general model for studying multi-timescale optimal control. We consider a system subject to linear dynamics which is perturbed by noise; we make absolutely no assumptions about the nature of the noise, i.e., it may be random or even adversarial. This system can be controlled by two controllers, one of which is a traditional, ``fast timescale" controller that can react immediately to the noise, and another which is a novel, ``slow timescale'' controller that is only able to react slowly, but which is empowered with access to more information than the fast controller and is potentially cheaper to use.

Secondly, we prove that one cannot expect to be able to design near-optimal controllers for multi-timescale control problems without the use of predictions. Our proof technique is based on a blackbox reduction to online convex optimization, a problem that has been intensively studied within the online algorithms community over the past decade. We use this reduction to describe a novel algorithm for the classic, fast timescale problem, a result which is of interest in its own right.

Thirdly, we introduce a new multi-timescale control policy, \name, and derive strong guarantees on its performance. In particular, we prove that the per-step cost incurred by our algorithm is at most a constant more than that incurred by the offline optimal.
The design of our policy is motivated by a structural result about the offline optimal control action, which highlights a strong decomposition between fast and slow timescale controllers. Applying this idea to the design of the online algorithm, we are able to achieve a clean separation between timescales. Remarkably, our decomposition results in a purely reflexive, ``dumb" fast controller, which performs no optimization or lookahead.  Thus, all of the 
computational burden is shifted onto the slow, ``smart'' controller. This property of \name\ is desirable in many applications since the slow controller is often centralized and able to take a global view of the system, but the fast controllers are decentralized and myopic, e.g., the power systems, networking, and robotics examples mentioned above. 

\subsection{Related literature}

This paper broadly falls into the category of optimal control \cite{zhou1996robust, bertsekas1995dynamic}. Typical methods for solving optimal control problem involve Pontryagin's principle \cite{rishel1965extended, raymond1999pontryagin} and Hamilton-Jacobi-Bellman equation \cite{bertsekas1995dynamic, borrelli2005dynamic}. With the rare exception of Linear Quadratic (LQ) systems, optimal control problems are generally nonlinear and do not admit analytical solutions. It is therefore necessary to solve optimal control problem via numerical methods. However, most existing numerical methods for optimal control (see \cite{rao2009survey} for a survey) do not scale well, and decomposing large scale problems into smaller subproblems is often required.   
There is large of literature on decomposition in the field of convex optimization and distributed computing.  Common approaches include primal-dual decomposition \cite{low1999, palomar2007alternative}, alternating direction method of multipliers \cite{boyd_admm, combettes2005signal} etc. These approaches have been crucial in developing distributed algorithms in various applications, e.g.,  communication networks \cite{srikant2012mathematics,low2002internet,chiang2007layering}, power systems \cite{erseghe2014distributed, peng2016distributed}, robotics \cite{suzuki1999distributed, raffard2004distributed}. However, these approaches are focused on spatial decomposition, and our focus in this paper is on temporal decomposition into independent controllers at different timescales.  

The most related prior work is \cite{matni2016}, which proposes an architectural decomposition of the optimal control problem into two layers: a top level trajectory planning problem that generates reference signals and  a low level  tracking problem that simply follows the reference points. However \cite{matni2016} does not provide optimality guarantees for the decomposition. Another related recent paper is  \cite{cai2015distributed}, which focuses on temporal decomposition in the context of power systems.  The work provides an optimality condition  for time-scale decomposition of optimal control in power systems.  But, note that \cite{cai2015distributed} considers a problem without dynamics. In this paper, we propose timescale decomposition for a general optimal control problem with linear dynamics and we provide provable performance guarantees. 

\section{MODEL}
\label{s.model}

Our goal in this paper is to study the design of controllers for systems that operate at multiple timescales.  To this end, we focus on a simple but general optimal control problem.  

The multi-timescale problem we consider builds on the following optimal control problem, which operates at a single timescale:
\begin{align}
\min_{x, f}  \quad & \sum_{t = 1 }^{T} c_x(x_t) + c_f(f_t) \label{e.fastonly}\\ 
\text{s.t.} \quad & x_t = Ax_{t-1} + B^f f_t + w_t \nonumber \\ 
   & x_0 = 0 \nonumber 
\end{align}
Here $x_t \in \mathbb{R}^n$ is the state variable, $f_t \in \mathbb{R}^n$ is the control action and $w_t \in \mathbb{R}^n$ is the disturbance. In our technical results, we assume that the control matrix $B^f$ is invertible; considering the non-invertible case is an interesting direction for future work. The cost functions $c_x(\cdot), c_f(\cdot)$ are usually assumed to be non-negative and convex. The special case when each noise increment $w_t$ is an i.i.d.\ Gaussian random variable and $c_x(\cdot), c_f(\cdot)$ are positive definite quadratic forms represents the Linear Quadratic Regulator (LQR) framework \cite{chow1975analysis,kwakernaak1972linear, sontag2013mathematical}.

To extend \eqref{e.fastonly} to a multi-timescale control problems, we introduce a ``slow'' controller.  The slow controller reacts much less quickly to noise than the fast controller; however there are two potential benefits afforded by the existence of the slow controller. 

First, in many situations the slow controller is centralized, and hence can use global information to make make better decisions than the decentralized, localized fast controllers. In our context, we model this by allowing the slow controller access to predictions of future noise increments. An example where the slow controller has this benefit is software defined networking, where the centralized controller has access to much more information than the local distributed controllers  that provide congestion control via simple reactive policies \cite{kreutz2015software}.  Similarly, this type of interaction between a ``smart'' slow controller and a ``reflexive'' fast controller is common in robotics \cite{espenschied1996biologically}.

Second, in many cases the slow controller is much cheaper to operate than the fast controller.  Thus, making use of the slow controller is crucial for minimizing cost.  For example, in the smart grid cheap ``baseload'' generators are used to supply the bulk of demand, whereas fast and relatively expensive ``peaker" generators are used to quickly correct any imbalances between supply and demand that may arise \cite{senjyu2003fast, peters2007basics, masters2013renewable}, \cite{cai2015distributed}. A similar distinction happens between economic dispatch and frequency regulation.  In fact, a motivation for this paper comes from recent work in \cite{cai2015distributed} that highlights a timescale separation between these controllers. 

Adding a slow controller to the optimal control problem in \eqref{e.fastonly} gives:
\begin{align}
   \min_{x, f, s} \quad & \sum_{t = 1}^{T} c_x(x_t) + c_f(f_t) +  c_s(s_t) \label{e.globalopt}\\ 
\text{s.t.} \quad &  x_t = Ax_{t-1} + B^f f_t + B^s s_t + w_t \nonumber \\ 
   &x_0 = 0 \nonumber \\ 
   &s_t = s_{t-1} \hspace{5mm} \forall t \not \in 0, k, 2k, \ldots \nonumber
\end{align}
Here $s_t$ denotes the control action of a slow controller. The constraint on $s_t$ means that the slow controller cannot react quickly, i.e., it can only change its action every $k$ timesteps. 

This formulation leads to a intrinsic notion of timescales: there is a \textit{fast timescale} consisting of the timesteps $\{ 1, 2, \ldots \}$ in which the fast controller reacts, and a \textit{slow timescale} consisting of the timesteps $\{1, k + 1, 2k + 1, \ldots \}$ at which the slow controller reacts. Clearly one could continue to add other timescales to this formulation as well, but we focus on the two timescale case for clarity.  

The focus of this paper is the design of a set of fast timescale and slow timescale controllers that operate independently but, together, approximate the optimal value of \eqref{e.globalopt} without fully knowing the $w_t$'s in advance. Motivated by the applications mentioned above, our goal is to develop designs where the slow controller is sophisticated and predictive, but the fast controller is simple and reactive.  

One of the key differences of our approach compared to classical control theory lies in how we measure performance. Typical results in the control theory literature focus on settings with distributional assumptions about the noise vector $w$ (for example, i.i.d.\ Gaussian), and seek to minimize the expected cost with respect to this distribution. In contrast, we use the approach of the online algorithms community and analyze the worst-case performance without distributional assumptions via the \textit{competitive ratio} \cite{karp1992line, fiat1998online}.

Formally, the competitive ratio is defined as follows.  Let $OPT$ denote the optimal value of \eqref{e.globalopt} and $ALG$ the cost incurred by a specific algorithm.  Then, the \textit{competitive ratio} is defined as $$CR(ALG) = \sup_w \frac{ALG}{OPT}.$$ This quantity measures the worst-case performance of an algorithm relative to the offline optimal and, in particular, makes no distributional assumptions on $w$. An algorithm is said to be  \textit{constant competitive} if its competitive ratio is bounded by a finite constant, independent of $T$. 

We show in Section \ref{s.fastonly} that it is impossible to design constant competitive algorithms for \eqref{e.globalopt} without using predictions of the future $w_t$. For this reason, our results focus on settings where algorithms have access to a limited number of noisy predictions of future $w_t$.  In particular, we assume that, at the start of each slow timescale interval, we have estimates $\hat{w}_t$ of the true noise increments over that slow timescale interval. Importantly, we do not make distributional assumptions about the predictions or prediction errors.  

Finally, one note on notation: throughout this paper, we follow the standard convention that vector valued variables are lowercase and matrix valued values are uppercase. When we write $\| A \|_a$, we mean the matrix norm of $A$ induced by the vector norm $\| \cdot \|_a$. Also, we follow the convention of the algorithms community and often abuse notation to let an algorithm's name denote the cost it incurs.

\section{Hardness of Multi-Timescale Control } \label{s.fastonly}



Before turning to the design and analysis of an online algorithm for multi-timescale control, it is natural to ask what performance we should expect to be able to attain.  In particular, should we expect to be able to find a constant competitive algorithm?

We show in this section that the answer is ``no'' in general, but that it becomes ``yes" when the algorithm has access to a limited number of noisy predictions.  This observation is crucial to the design and analysis of the algorithm we present in Section \ref{s.multi}. 

Interestingly, we cannot even expect to be able to design a constant competitive algorithm for the fast control subproblem of \eqref{e.globalopt} given by \eqref{e.fastonly}.  To show this, we prove below that \eqref{e.fastonly} can be reformulated as an online convex optimization problem, which is a classical online algorithms problem that has received considerable attention in the last decade \cite{lin2012,lin2013,antoniadis2016chasing,chen2015,chen2016}.  Importantly, competitive algorithms for online convex optimization algorithms do not exist in general, unless the algorithms are given access to noisy predictions about the future. 

Formally, the equivalence to online convex optimization is stated as follows.


\begin{prop} \label{prop.oco}
Suppose both $c_s(\cdot)$ and $c_f(\cdot)$ are a norm, $\| \cdot \|$, and suppose $B^f = B^s$.  Then \eqref{e.fastonly} is equivalent to \eqref{e.globalopt} and, further, \eqref{e.fastonly} can be reformulated as 
\begin{align}
  \min_{y}\sum_{t = 1}^{T} c_t(y_t) + \|(B^f)^{-1}(y_t - Ay_{t-1})\|, \label{e.ocoformulation}
\end{align}
where $y_0 = 0$ and $c_t(y_t) = c_x(y_t + v_t)$ for some $v_t$.
\end{prop}

\begin{proof}
First, we need to argue that the multi-timescale problem \eqref{e.globalopt} is equivalent to the fast timescale problem in \eqref{e.fastonly} under the assumptions of the proposition.  Suppose $f^*$ is an optimal control action for  $\eqref{e.fastonly}$. Then the pair $(f^*, 0)$ is an optimal pair of fast and slow control actions for \eqref{e.globalopt}, since it is feasible and achieves the same cost. Conversely, suppose $(f^*, s^*)$ is an optimal pair of control actions for  $\eqref{e.globalopt}$. Then the action $f^* + s^*$ is an optimal action for \eqref{e.fastonly}, by the same reasoning.


The second part of the proposition is to prove the reformulation of \eqref{e.fastonly} as \eqref{e.ocoformulation}.  To do this, we can iterate the dynamics and apply a change of variables. Specifically, iterating the dynamics in \eqref{e.fastonly} backwards in time, we see that $$x_t = \sum_{i = 1}^{t} A^{t-i}B^f f_i + \sum_{i = 1}^{t} A^{t-i}w_i.$$
Next, introduce the change of variables $$y_t = \sum_{i = 1}^t A^{t-i}B^f f_i, v_t = \sum_{i = 1}^t A^{t-i}w_i.$$ This yields \eqref{e.ocoformulation}. Notice that, given the solution to \eqref{e.ocoformulation}, we can construct the corresponding solution to \eqref{e.fastonly} by setting $f_t = (B^f)^{-1}(y_t - Ay_{t-1}).$
\end{proof}

Problem \eqref{e.ocoformulation} is a specific kind of online convex optimization problem known as a ``Smoothed'' Online Convex Optimization (SOCO). Specifically, convex cost functions $c_t(y_t) = c_x(y_t + v_t) $ arrive online and the goal of the online algorithm is to minimize the cost paid by choosing the sequence of actions $\{y_t\}$. The term $\|(B^f)^{-1}(y_t - Ay_{t-1})\|_f$ acts as a regularizer, penalizing choices that differ from the previous choice $y_{t-1}$ under the dynamics of $A$. Proposition \ref{prop.oco} provides a reduction from SOCO to \eqref{e.globalopt}: if we had a competitive online algorithm for \eqref{e.globalopt}, we would have one for SOCO as well.

SOCO problems have been intensely studied in the past decade due to their widespread applications in fields as diverse as motion tracking \cite{ledata}, power management for large data centers \cite{lin2013}, geographical load-balancing for Internet scale applications \cite{lin2012} \cite{qureshi2009cutting}, and video streaming \cite{niu2011}, \cite{joseph2012jointly}.  In general, while there exist constant competitive algorithms for one dimensional \cite{bansal2015}, \cite{lin2012} and two dimensional \cite{antoniadis2016chasing} SOCO problems, it is unknown whether there exist constant competitive algorithms for higher dimensions.  Further, it has been shown that SOCO problems are  equivalent to Convex Body Chasing \cite{friedman1993convex} in the sense that a competitive algorithm for one implies the existence of a competitive algorithm for the other \cite{antoniadis2016chasing}.  This highlights the difficulty of obtaining constant competitive algorithms since Convex Body Chasing has been open for several decades. 

Due to the difficulty of SOCO-style problems,  much of the work on these problems has focused on settings where the online algorithms have access to (possibly noisy) predictions about future cost functions. For example, given perfect lookahead in a prediction window of length $w$, there exist algorithms whose competitive ratio is $1 + O(1/w)$, independent of dimension \cite{lin2012}. Similar positive results are possible in cases with noisy predictions, e.g., \cite{chen2015,chen2016}.

Given Proposition \ref{prop.oco}, the positive results described above can yield effective algorithms for the single timescale, optimal control with linear dynamics in \eqref{e.fastonly}. To highlight this, we focus on a particularly promising algorithm from the SOCO literature called \textit{Averaging Fixed Horizon Control (AFHC)}.

AFHC was introduced in \cite{lin2012} and has since been studied in \cite{chen2015,chen2016,badiei2015online}.  AFHC is parameterized by the size of the prediction window it uses, which we denote by $w$. It works by averaging together the control actions of $w + 1$ independent Fixed Horizon Control (FHC) algorithms. The $k$-th FHC algorithm ($k = 1 \ldots w + 1$) starts at timestep $k$ by greedily choosing the set of control actions that minimize the cost over time $[k, k + w]$, and then repeatedly chooses control actions to minimize cost over each consecutive length $w + 1$ window. The control action output by AFHC is the average of the control actions of all $w + 1$ FHC algorithms.

In general, \cite{lin2012} proves that AFHC is $1+O(1/w)$ competitive for SOCO problems with costs bounded below by a positive constant $c_0$.  In the setting of this paper, we can prove a more precise result that highlights the impact of the structure of the dynamics.  

\begin{thm} \label{e.afhc}
Suppose each $c_t$ is $m$-strongly convex and bounded below by a positive constant $c_0$. Then the competitive ratio of AFHC for (3) is at most $$1 + \frac{\|(B^f)^{-1}A\|^2}{2m(w+1)c_0}$$ and in particular is $1 + O(1/w)$. 
\end{thm}

This result highlights the ability of AFHC to perform well in a classic control problem, even in the \emph{adversarial} setting. Further, the bound in the theorem highlights the impact of the structure of the dynamics on the performance of the algorithm. In general, as $w$ tends to infinity the competitive ratio of AFHC will tend to one. This is unsurprising, since the algorithm will have access to more and more information about the future and hence will be able to make better decisions. However, Theorem \ref{e.afhc} shows that even when $w$ is small, AFHC can attain near optimal performance provided $\|(B^f)^{-1}A\|$ is sufficiently small. The matrix $A$ can be interpreted as the gain of the system dynamics, and the matrix $B^f$ as the gain of the fast controller; hence the expression $\|(B^f)^{-1}A\|$ is intuitively a measure of the fast controllers ability to counteract the gain of the system. 

\begin{proof} Let $\Omega_k = \{k, k + w, k + 2w, \ldots \}$ be the set of times when the $k$-th FHC algorithm recomputes its control trajectory. Let $y^*$ denote the optimal trajectory for (3) and $y^k$ denote the choice of the $k$-th FHC algorithm. We define a function which measures the total cost incurred by a trajectory over $[s, s + w]$, starting from the point $y^k_{s-1}$: 

\begin{align}
g_{s, s+ w}(y) =& \sum_{t = s}^{s + w} c_t(y_t) + \sum_{t = s + 1}^{s + w} \|(B^f)^{-1}(y_t - Ay_{t-1} )\|  \nonumber \\
&+ \|(B^f)^{-1}(y_s - Ay^k_{s-1})\| \nonumber
\end{align}
Notice that $g_{s, s+w}$ is itself $m$-strongly convex; it is the sum of the $m$-strongly convex cost functions and the convex switching costs. Hence for all $s$ we have 

\begin{align}
g_{s, s+ w}(y^k) - g_{s, s+ w}(y^*) \leq & \nabla g_{s, s+w}(y^k)^T(y^k - y^*) \nonumber \\
&-\frac{m}{2}\sum_{t = s}^{s + w} \|y^k_t - y^*_{t}\|^2 \nonumber
\end{align}
Notice that for all $s \in \Omega_k$, the gradient term vanishes, since by definition the $k$-th FHC algorithm chooses a trajectory which minimizes $g_{s, s + w}$ at each $s \in \Omega_k$. Letting $d_t = \|y^k_t - y^*_{t} \| $ and summing up over all $s \in \Omega_k$ gives: $$\sum_{s \in \Omega_k} g_{s, s+ w}(y^k) - \sum_{s \in \Omega_k} g_{s, s+ w}(y^*) \leq -\frac{m}{2}\sum_{s \in \Omega_k} \sum_{t = s}^{s + w} d_{t}^2 $$
The first sum on the left hand side is the total cost incurred by the $k$-th FHC algorithm, which we denote by $FHC^k$. Using the definition of $g_{s, s+w}$ and the reverse triangle inequality, we can bound the second term: $$\sum_{s \in \Omega_k} g_{s, s+ w}(y^*) \leq OPT + \|(B^f)^{-1}A \|\sum_{s \in \Omega_k} d_{s-1}$$ from which we obtain $$FHC^k - OPT \leq \sum_{s \in \Omega_k}  \|(B^f)^{-1}A \| d_{s-1} -\frac{m}{2} d_{s-1}^2 $$
Here we used the fact that $-\frac{m}{2}d_t^{2} $ is always nonpositive, so throwing away some of these terms can only increase the righthand side.
Maximizing the summands  in $d_{s-1}$, we obtain $$FHC^k - OPT \leq \sum_{s \in \Omega_k} \frac{\|(B^f)^{-1}A\|^2}{2m} $$ 
We average all $w + 1$ FHC algorithms and apply Jensen's Inequality to obtain $$AFHC - OPT \leq  \frac{1}{w +1}\sum_{t = 1}^T  \frac{\|(B^f)^{-1}A\|^2}{2m} $$
Finally we divide by $OPT$ and use the bound $OPT \geq c_0T$ to get a bound on the competitive ratio: $$1 + \frac{\|(B^f)^{-1}A\|^2}{2m(w+1)c_0} $$
which establishes the $1 + O(1/w)$ claim. 
\end{proof}

\section{Architectural Decomposition for Multi-Timescale Control} \label{s.multi}

We now turn our attention to the joint multi-timescale control problem in \eqref{e.globalopt}, and focus on the co-design of fast and slow controllers. Recall that, while the slow controller cannot act as frequently, there are two benefits it usually provides: (i) it may have more information and computational power than the fast controller, e.g., in software defined networking and robotics, and (ii) it may be cheaper to operate than the fast controller, e.g., when scheduling generation in the smart grid. To capture these benefits of a slow controller, we consider a setting where the slow controller has access to noisy predictions but the fast controller does not.  We also specifically highlight the case where the slow controller is cheaper to operate, though our results apply more generally. 

Our main result in this section provides a performance bound for a new, near-optimal algorithm -- \emph{Multi-timescale Reflexive Predictive Control} (\name) -- that consists of a simple, reflexive fast timescale controller and a predictive slow timescale controller. For concreteness and ease of presentation we focus on the case where the cost functions $c_x, c_s, c_f$ in \eqref{e.globalopt} are norms $\| \cdot \|_x, \| \cdot \|_s, \| \cdot \|_f$.

\subsection{An overview of \name}

Informally, \name\ works as follows.  Over each slow timescale slot, the slow controller greedily plays the slow control action which minimizes the expected cost using the predictions $\hat{w}_t$, under the assumption that the fast controller will keep the state at zero. As the true noise increments $w_t$ are revealed one by one, the fast controller myopically corrects any noise so as to keep the state at zero.

Formally, let $\hat{f}$ and $\hat{s}$ denote the fast and slow control actions of \name. Then, the operation of each is as follows: 
\begin{align} 
\hat{s}_r& = \min_s \left[ k \|s_r\|_s + \sum_{t =r }^{r + k -1} \|(B^f)^{-1}(B^ss_r + \hat{w}_t)\|_f \right] \label{e.slowaction} \\
\hat{f}_t &= -(B^f)^{-1}(B^s\hat{s}_r + w_t) \hspace{6mm} t = r, \ldots {r + k - 1}  \label{e.fastaction}
\end{align}

Notice that the fast controller is very simple; it uses no predictions and performs no optimization. All of the prediction and optimization is shifted onto the slow controller. This is consistent with how the two controllers are used in many applications, where the slow controller is often centralized, with access to global information, but the fast controllers are usually decentralized, localized, and computationally limited. For example, in the smart grid a slow timescale global optimization problem is solved (economic dispatch) and then localized fast timescale controllers myopically correct any deviations that may arise (frequency regulation).


\subsection{Performance of \name}

Our main technical result is a performance bound for \name.  In particular, the following result shows that, despite the difficulty of the multi-timescale control problem, \name\ maintains a per-stage cost within a constant factor of the offline optimal, even when adversarial inputs are considered.  

\begin{thm} \label{t.multitimescale_cr} Assume the cost functions $c_x, c_s, c_f$ in \eqref{e.globalopt} are norms $\| \cdot \|_x, \| \cdot \|_s, \| \cdot \|_f$. Then
\name\ has an average per-stage cost within a constant factor of optimal.  Specifically, 
\begin{align*}
\frac{MRPC}{T} \leq & \max{\left(\frac{\left( 1 + \| A \|_x \right)\|(B^f)^{-1}\|}{c}, 1 \right)} \frac{OPT}{T} \\ 
& \quad\quad\quad\quad + 2\|(B^f)^{-1}\| 
E(\hat{w},w)   \nonumber 
\end{align*}
where $c$ is a constant such that $\| v \|_x \geq c \| v \|_f$ for all $v$ and $E(\hat{w}, w)$ is the sample path average prediction error:
$$E(\hat{w}, w) = \frac{1}{T} \sum_{t = 1}^T \|\hat{w}_t - w_t \|_f$$
\end{thm}

Before moving to the proof, let us make a few remarks about Theorem \ref{t.multitimescale_cr}.  To begin, recall that even the single timescale problem could not be solved optimally by a fast timescale controller, and so the performance bound in Theorem \ref{t.multitimescale_cr} is surprisingly strong, especially given that the fast time scale controller in \name\ does not use any predictions -- it is simply reflexive.  

To get intuition for the bound itself, let us first look at the second term.  The second term in the bound corresponds to the inefficiency due to noisy predictions. In particular, if we assume perfect lookahead (i.e $\hat{w}_t = w_t$ for all $t$), then the second term disappears.  Thus, we see that prediction error has only an additive effect. It is important to realize that the analysis makes no modeling assumptions on the form of the prediction error.  The error can be adversarial or stochastic and the result still holds. 

The first term bounds the per-step cost incurred by our algorithm relative to the per-step cost incurred by the offline optimal.  To get intuition for it,  consider the case where control costs dominate the state costs.  Specifically, consider the case where $c \geq 2 \|(B^f)^{-1} \|$, and there are no errors in predictions. In this case, we have $MRPC = OPT$. It is worth highlighting this result in words: \textit{when state costs dominate control costs and prediction errors are small, our distributed algorithm achieves the optimal value of \eqref{e.globalopt}}. This is remarkable, since the offline optimal has a formidable advantage compared to our online algorithm - it knows the full noise vector $w$ in advance, whereas during each slow timescale interval our online algorithm only has access to predictions about the noise in that interval.

Finally, it is important to note that Theorem \ref{t.multitimescale_cr} is incomparable to Theorem \ref{e.afhc} since Theorem \ref{t.multitimescale_cr} compares to the offline optimal of the multi-timescale problem while Theorem \ref{e.afhc} compares to the offline optimal of a single stage problem.  In settings where the slow timescale controller is much cheaper than the fast timescale controller the cost difference between these problems can be arbitrarily large. 

We now move to the proof of Theorem \ref{t.multitimescale_cr}.  The proof is technical, but also provide crucial intuition into the form of \name. In particular, the proof includes a lower bound on the offline optimal which motivates the decomposition between the fast and slow timescales used in the design of \name.  It is this bound that highlights the ability to obtain timescale separation via the optimization decomposition in \name.

\subsection{Proof of Theorem \ref{t.multitimescale_cr}} 

Recall, that we are considering the case where the convex cost functions in \eqref{e.globalopt} are norms. Further, it is convenient to absorb the constraint on the slow controller directly into the objective and rewrite \eqref{e.globalopt} as 
\begin{align}
  \min_{x, f, s} \quad & \sum_{r \in \mathcal{S}} \left[ k \|s_r\|_s + \sum_{t =r }^{r + k -1} \|x_t\|_x + \|f_t\|_f \right] \label{e.globalopt2} \\
\text{s.t.} \quad &  x_t = Ax_{t-1} + B^ff_t + B^ss_t + w_t \nonumber \\
   & x_0 = 0 \nonumber 
\end{align}
Here $\mathcal{S} = \{1, k + 1, 2k + 1, \ldots  \}$ is the set of slow timescale steps, i.e., when the slow controller can change its control action.

The first step in the analysis is to establish a lower bound on the cost incurred by the offline optimal. As mentioned above, this lower bound highlights the decomposition between fast and slow used in the design of \name.

To prove the lower bound we make use of the following technical lemma.

\begin{lemma} \label{e.lowerlemma}
Let $v \in \mathbb{R} ^n$, and let $M \in \mathbb{R}^{n \times n}$ be an invertible matrix. Let $\| \cdot \|_a, \| \cdot \|_b$ be any two norms on $\mathbb{R}^n$, and let $c$ be a constant such that $\|v\|_a \geq c \|v\|_b$ for all $v$. 
For all $\alpha, \beta > 0$ we have
\[\min_{x}{ \alpha \|v + Mx\|_a + \beta \|x\|_b } \geq \min{\left( \frac{\alpha c}{\|M^{-1}\|_b}, \beta   \right) \|M^{-1} v\|_b}.
\]

\label{lemma: norm_minimization}
\end{lemma}

\begin{proof}
We have: 
\begin{align*}
&\min_{x}{ \alpha \|v + Mx\|_a + \beta \|x\|_b } \\
\ge &\min_{x}{ \alpha c  \|M (x + M^{-1}v)\|_b + \beta  \|x\|_b } \\
\ge &\min_{x}{  \frac{\alpha c}{\|M^{-1}\|_b}\|x + M^{-1}v\|_b + \beta  \|x\|_b } \\
\geq &\min_{x} \frac{\alpha c}{\|M^{-1}\|_b} \Big| \| x\|_b - \|M^{-1} v\|_b \Big| + \beta  \|x\|_b 
\end{align*}
The first inequality follows from the equivalence of norms in finite dimensional linear spaces. The second inequality is because $\| y \| = \|M^{-1} M y\| \le \| M^{-1}\|\|My\|$, hence $\|My\|\ge \frac{1}{\|M^{-1}\|}\|y\|$ for all $y$. The last inequality is just the reverse triangle inequality. 


The last optimization is an optimization over the scalar variable $\| x\|_b$ and it is easy to see that it is lower bounded by $$ \min{\left( \frac{\alpha c}{\|M^{-1}\|}, \beta   \right)\|M^{-1}v\|_b}. $$
\end{proof}
Now we are ready to prove the lower bound. 

\begin{lemma} \label{l.loweropt}
Letting $OPT$ denote the optimal solution for \eqref{e.globalopt2}, we have:
\begin{align} 
 OPT \geq \min_{s} \sum_{r \in \mathcal{S}} \left[  k  \|s_r\|_s +   C\sum_{t =r }^{r + k -1}   \| (B^f)^{-1}(B^s s_t + w_t)\|_f  \right] \nonumber 
\end{align}
where $$C = \min{\left(\frac{c}{ \left( 1 + \| A \|_x \right)\|(B^f)^{-1}\|}, 1\right)}$$ and $c$ is a  constant such that $\|v\|_x \geq c\|v\|_f$ for all $v$.
\end{lemma}
\begin{proof}
Suppose $\hat{x}, \hat{f}, \hat{s}$ are some arbitrary feasible choices of the decision variables in \eqref{e.globalopt2}, which incur the associated cost $COST$. We have 
\begin{align}
COST =& \sum_{r \in \mathcal{S}}\sum_{t =r }^{r + k -1}  \|\hat x_t\|_x +  \|\hat f_t \|_f +  \|\hat s_r \|_s \nonumber \\
=&  \sum_{r \in \mathcal{S}}\sum_{t =r }^{r + k -1}  \|A\hat x_{t-1} + B^f\hat{f_t} + B^s\hat{s_r} + w_t\|_x \nonumber \\
& \hspace{13mm} +  \|\hat{f_t}\|_f +  \|\hat s_t\|_s   \nonumber \\
 \geq &   \sum_{r \in \mathcal{S}}\sum_{t =r }^{r + k -1} \Big[ \| B^f\hat{f_t} + B^s\hat{s_r} + w_t\|_x - \|A \|_x \| \hat{x}_{t-1} \|_x \nonumber \\
& \hspace{13mm} +  \|\hat{f_t}\|_f +  \|\hat{s_r}\|_s \Big] \nonumber \\
\geq & - \|A\|_x COST + \sum_{r \in \mathcal{S}}\sum_{t =r }^{r + k -1}  \| B^f\hat{f_t} + B^s\hat{s_r} + w_t\|_x   \nonumber \\
& \hspace{25mm} +  (1 + \| A \|_x) \left(\|\hat{f_t}\|_f +  \|\hat{s_r}\|_s \right)  \nonumber
\end{align}
from which we obtain the lower bound on $COST$ given by
{\small
\begin{align}
\sum_{r \in \mathcal{S}} \left[ k  \|\hat s_r\|_s + \sum_{t =r }^{r + k -1}  \beta \| B^f\hat f_t + B^s\hat s_r + w_t\|_x +  \|\hat f_t\|_f \right]  \nonumber 
\end{align}
}
\noindent where $\beta = \frac{1}{1 + \| A \|_x}$. Since $\hat{x}, \hat{f}, \hat{s}$ were arbitrary feasible values, and in particular, could be taken to be the optimal values for $\eqref{e.globalopt2}$, we obtain a lower bound on $OPT$ given by
{\small
\begin{align}
& \min_{f, s} \sum_{r \in \mathcal{S}} \left[ k  \|\hat s_r\|_s +  \sum_{t =r }^{r + k -1}  \beta \|B^f f_t + B^s s_r + w_t\|_x +  \|f_t\|_f \right].  \nonumber
\end{align}
}
Examining the structure of this expression, we observe that once each $s_r$ is fixed, the resulting optimization in $f$ resembles that in Lemma \ref{e.lowerlemma}, which leads directly to the theorem.
\end{proof}

The lower bound has the following interpretation.  Suppose the state is set at zero.  After the slow controller has set its action to be $s_r$, the fast control action which corrects the remaining deviation from zero is $(B^f)^{-1}(B^s s_r + w_t)$, and our lower bound is the sum of the resulting costs (up to the constant $C$). Notice that the fast controller is extremely simple - all it does is continually correct any residual noise so that the state is always kept at zero. This is a crucial observation: \textit{the form of the lower bound highlights a clear separation between a ``smart'', slow controller that does the planning and a ``dumb'' reactive fast controller}. This separation is then what we mimic in the design of \name, and also guides our analysis of the algorithm, as is evident in the following lemma, which  provides an upper bound on the cost of \name. 

\begin{lemma} \label{l.upperalg}
{\small 
\begin{align}
MRPC \leq& \min_{ s} \sum_{r \in \mathcal{S}} \left[ k \|{s_r}\|_s +  \sum_{t =r }^{r + k -1}  \|(B^f)^{-1}(B^s s_r + w_t)\|_f \right] \nonumber \\
	&+ 2\|(B^f)^{-1}\|\sum_{t = 1}^T \|\hat{w_t} - w_{t} \|_f \nonumber 
\end{align}
}
\end{lemma}

\begin{proof} Plugging our control actions into the cost function and applying the Triangle Inequality, we have 
\begin{align}
MRPC =& \sum_{r \in \mathcal{S}} \left[ k \|\hat{s}_r\|_s + \sum_{t =r }^{r + k -1} \|(B^f)^{-1}(B^s\hat{s}_r + w_t)\|_f \right] \nonumber \\
\leq& \sum_{r \in \mathcal{S}} \left[ k \|\hat{s}_r\|_s + \sum_{t =r }^{r + k -1} \|(B^f)^{-1}(B^s\hat{s}_r + \hat{w}_t)\|_f \right] \nonumber \\ 
&+ \sum_{t = 1}^T \|(B^f)^{-1}(\hat{w}_t - w_t) \|_f \nonumber 
\end{align}
Now we use the definition of $\hat{s}_r$ to obtain the upper bound
\begin{align}
 \min_{ s} \sum_{r \in \mathcal{S}} \left[ k \|{s_r}\|_s +  \sum_{t =r }^{r + k -1} \|(B^f)^{-1}(B^s s_r + \hat{w}_t)\|_f \right] \nonumber \\
+ \|(B^f)^{-1}\|\sum_{t = 1}^T \|\hat{w_t} - w_{t} \|_f \nonumber 
\end{align}
Notice that the minimization in $s$ depends on the estimates $\hat{w}_t$, not the true values $w_t$. Applying the Triangle Inequality once again allows us to produce an upper bound where the minimization is over the true values:
\begin{align} \min_{ s} \sum_{r \in \mathcal{S}} \left[ k \|{s_r}\|_s +  \sum_{t =r }^{r + k -1} \|(B^f)^{-1}(B^s s_r + w_t)\|_f \right] \nonumber \\
+ 2\|(B^f)^{-1}\|\sum_{t = 1}^T \|\hat{w_t} - w_{t} \|_f \nonumber 
\end{align}
This proves the claim.
\end{proof}

The combination of Lemmas \eqref{l.loweropt} and \eqref{l.upperalg} immediately yield Theorem \ref{t.multitimescale_cr}.

\section{Concluding Remarks}
In this paper we present a simple and general model of multi-timescale control problems. We prove a hardness result using a blackbox reduction to online convex optimization, and show that predictions are necessary to construct a constant competitive algorithm. Further, we propose a simple control policy with a clean separation between timescales that uses only a small number of noisy predictions. 

Our decomposition results in a sophisticated, predictive slow controller and a simple, reactive fast controller. This framework mirrors the architecture of many real-world control systems, where a slow, centralized controller guides the system towards global optimality while fast, decentralized controllers help to quickly counteract any perturbations that may arise. Remarkably, despite the simplicity of our fast controller and the fact that our policy has access to only limited information about the future, we derive strong guarantees on the performance of our policy. In particular, we prove that the per-step cost incurred by our algorithm is at most a constant more than that incurred by the offline optimal, and in some cases our policy even matches the offline optimal costs. 

There are several natural and important problems left open by our work. Firstly, we do not consider delay in our model, though many real-world systems feature information-sharing constraints arising from delay. It would be natural to add such constraints to the slow controller. Secondly, it would be interesting to consider a hybrid model that is distributed across both time and space, i.e one that features both decentralization and multiple timescales. Most systems that operate across multiple timescales, such as the smart grid, also feature both centralized and localized controllers. 

\bibliographystyle{abbrv}
\bibliography{references}

\end{document}